\documentclass[11pt]{amsart}

\usepackage{graphicx,psfrag}
\usepackage{amssymb,amsmath,amsthm}
\usepackage{enumerate}  
\usepackage{a4wide}
\usepackage{blindtext}

\usepackage{caption}
\usepackage{subcaption}

\usepackage{mathrsfs}
\usepackage{enumitem}

\numberwithin{equation}{section}
\allowdisplaybreaks[4]

\newtheorem{proposition}{Proposition}[section]
\newtheorem{theorem}[proposition]{Theorem}
\newtheorem{lemma}[proposition]{Lemma}
\newtheorem{corollary}[proposition]{Corollary}
\newtheorem{definition}[proposition]{Definition}
\theoremstyle{definition}
\newtheorem{remark}[proposition]{Remark}

\title{Repeated quasi-integration on locally compact spaces}
\author{S. V. Butler}
\address{Department of Mathematics,  University of California, Santa Barbara,  552 University Rd, Isla Vista, CA 93117, USA}
\email{svtbutler@ucsb.edu}
\date{February 18, 2019}

\begin{document} 

\begin{abstract}
When $X$ is locally compact, a quasi-integral (also called a quasi-linear functional) on $ C_c(X)$ is a homogeneous, positive functional that is only assumed 
to be linear on singly-generated subalgebras. We study simple and almost simple quasi-integrals, i.e., quasi-integrals whose 
corresponding compact-finite topological measures assume exactly two values. We present a criterion 
for repeated quasi-integration (i.e.,  iterated integration with respect to topological measures)
to yield a quasi-linear functional.  We find a criterion for a double quasi-integral to be simple. 
We describe how a product of topological measures acts on open and compact sets. 
We show that different orders of integration in repeated quasi-integrals give the same quasi-integral if and only if the
corresponding topological measures are both measures or one of the corresponding topological measures 
is a positive scalar multiple of a point mass. 
\end{abstract}
\maketitle
\medskip\noindent

\section{Introduction}

Quasi-integrals (also called quasi-linear functionals) generalize linear functionals. 
They were first introduced on a compact space by
J. Aarnes in \cite{Aarnes:TheFirstPaper}, where Aarnes also studied their properties, and 
proved the representation theorem of quasi-integrals in terms of topological measures (formerly called quasi-measures). 
For uses of quasi-linear functionals in symplectic topology one may consult 
numerous papers beginning with \cite{EntovPolterovich} (which has been cited over 100 times), and a monograph \cite{PoltRosenBook}. 

In \cite{Grubb:Products} D. Grubb gives a nice treatment of repeated quasi-integration on a product of compact spaces. 
In this paper we generalize all results from \cite{Grubb:Products} to a locally compact setting, and also obtain new results 
previously not known in the compact case. 
In proving results  about quasi-linear functionals and topological measures in a locally compact setting one loses 
certain convenient features present in the compact case: quasi-integrals and topological measures are not finite in general,     
subalgebras do not contain constants, quasi-integrals are no longer states, etc. On another level, to tackle the problems related to 
repeated integration on locally compact spaces, one needs a variety of results about quasi-liner functionals and topological measures. 
These include a Representation Theorem showing that quasi-linear functionals are obtained by integration with 
respect to compact-finite topological measures, continuity of quasi-integrals with respect to the topology of uniform convergence 
on compacta, and others. 
These results are obtained in \cite{Alf:ReprTh}, \cite{Butler:TMLCconstr},  \cite{Butler:QLFLC},
and  \cite{Butler:DTMLC}.  
         
In this paper $X$ is a Hausdorff, locally compact, connected space. In Section \ref{prelim} we give necessary definitions and facts. 
We define quasi-integrals and topological measures on a locally compact space and outline the correspondence between them.  
In Section \ref{Almsim} we study simple and almost simple quasi-integrals and topological measures, 
which are necessary for investigating repeated quasi-integration. In Section \ref{Products}
we define and study repeated quasi-integrals. In particular, we give a criterion for repeated quasi-integration 
to yield a quasi-linear functional and a criterion for a double quasi-integral to be simple. 
In Section \ref{Fubini} we give formulas that describe how a product of compact-finite topological measures acts on open and compact sets. 
We show that different orders of integration in repeated quasi-integrals give the same quasi-linear functional 
if and only if corresponding compact-finite topological measures are both measures or one of the corresponding topological measures 
is a positive scalar multiple of a point mass. 

\section{Preliminaries} \label{prelim}

We use continuous functions with the uniform norm; $C_c(X)$ is the set of real-valued continuous functions on $X$ with compact support, and
$C_c^+(X)$ is the set of nonnegative functions from $C_c(X)$.  
By $ supp \  f $ we mean $ \overline{ \{x: f(x) \neq 0 \} }$.  We denote by $1$ the constant function $1(x) =1$,
by $id$ the identity function $id(x) = x$, and by $1_E$ the indicator function of the set $E$. 
When we consider maps into extended real numbers, they are not identically $\infty$. 
We denote by $\overline E$ the closure of a set $E$, and by $ \bigsqcup$ a union of disjoint sets.
$\mathscr{O}(X)$,  $\mathscr{C}(X)$, and  $\mathscr{K}(X)$  stand, respectively, for the collections of open, closed, and compact subsets of  $X $.

We are interested in repeated quasi-integrals, i.e. repeated quasi-integration 
with respect to topological measures on locally compact spaces.
Since $X \times Y$ is locally compact iff $X$ and $Y$ are locally compact,  
we assume that $X$ and $Y$ are locally compact.

We will use the following definitions and facts. 

\begin{definition}\label{TMLC}
A topological measure on a locally compact space $X$ is a set function
$\mu:  \mathscr{C}(X) \cup \mathscr{O}(X) \to [0,\infty]$ satisfying the following conditions:
\begin{enumerate}[label=(TM\arabic*),ref=(TM\arabic*)]
\item \label{TM1} 
If $A,B, A \sqcup B \in \mathscr{K}(X) \cup \mathscr{O}(X) $ then $ \mu(A\sqcup B)=\mu(A)+\mu(B).$
\item \label{TM2} 
If $U\in\mathscr{O}(X)$ then
$
\mu(U)=\sup\{\mu(K):K \in \mathscr{K}(X), \  K \subseteq U\}.
$
\item \label{TM3}
If $F \in \mathscr{C}(X)$ then
$
\mu(F)=\inf\{\mu(U):U \in \mathscr{O}(X), \ F \subseteq U\}.
$
\end{enumerate}
A topological measure $\mu$ on $X$ is called compact-finite if $\mu(K) < \infty$ for $K \in \mathscr{K}(X)$;  $ \mu$ is finite if $\mu(X) < \infty$.
\end{definition}
 
\begin{theorem} \label{subaddit}
Let $\mu$ be a topological measure on a locally compact space $X$. 
The following are equivalent: 
\begin{itemize}
\item[(a)]
If $C, K$ are compact subsets of $X$, then $\mu(C \cup K ) \le \mu(C) + \mu(K)$.
\item[(b)]
If $U, V$ are open subsets of $X$,  then $\mu(U \cup V) \le \mu(U) + \mu(V)$.
\item[(c)]
$\mu$ admits a unique extension to an inner regular on open sets, outer regular Borel measure 
$m$ on the Borel $\sigma$-algebra of subsets of $X$. 
$m$ is a Radon measure iff $\mu$ is compact-finite. 
If $\mu$ is finite then $m$ is a regular Borel measure.
\end{itemize}
(See Theorem 34 in Section 4 of \cite{Butler:DTMLC}.)
\end{theorem}

\begin{remark} \label{smsubalg}
Let $X$ be locally compact.
For $ f \in C_c(X)$ we have $0 \in f(X) = [a,b]$. By a singly generated subalgebra of $ C_c(X)$ generated by $f$  we mean 
the smallest closed subalgebra of $C_c(X)$ containing $f$; it has the form 
\[ B(f) =  \{ \phi \circ f :  \phi(0) = 0, \ \phi \in C(f(X)) \}. \] 
When $X$ is compact, by a singly generated subalgebra of $C(X)$ generated by $f$  we mean 
the smallest closed subalgebra of $C(X)$ containing $f$ and $1$; it  has the form:
\[ A(f)  = \{ \phi \circ f :  \phi \in C({f(X)}) \}. \]
(See Remarks 4 and 2 in  \cite{Butler:QLFLC}).
\end{remark}

\begin{definition} \label{QI}
Let $X$ be locally compact. 
A quasi-integral (or a quasi-linear functional)  on $C_c(X)$ is a map $\rho : C_c(X) \longrightarrow \mathbb{R} $ such that:
\begin{enumerate}[label=(QI\arabic*),ref=(QI\arabic*)]
\item \label{QIconsLC}
$\rho$ is homogeneous, i.e. $\rho(a f) = a \rho(f)$ for $ a \in \mathbb{R}.$
\item \label{QIlinLC}
For each  $f \in C_c(X)  $ we have:
$\rho(g+h) =  \rho (g) + \rho (h)$ for $g,h$ in the singly generated subalgebra $B(f)$. 
\item  \label{QIpositLC}
$\rho$ is positive, i.e. $ f \ge 0 \Longrightarrow \rho(f)  \ge 0.$
\end{enumerate}
When $X$ is compact, we call $\rho$ a quasi-state if $\rho(1) =1$. 
\end{definition}

In this paper we are interested in quasi-integrals on $C_c(X)$ and  compact-finite topological measures for the reason given in the next remark. 

\begin{remark} \label{RemART}
There is an order-preserving isomorphism between compact-finite 
topological measures on $X$ and quasi-integrals on $C_c(X)$, and $\mu$ is a measure iff the 
corresponding functional is linear. 
See Theorem 42 in Section 4 of \cite{Butler:QLFLC} for this result and Theorem 3.9 in \cite{Alf:ReprTh} for the first version of 
the representation theorem.
We outline the correspondence.
\begin{enumerate}[label=(\Roman*),ref=(\Roman*)]
\item \label{prt1}
Suppose $\mu$ is a compact-finite topological measure on a locally compact space $X$, $f \in C_c(X)$. 
Then there exists a measure $m_f$ on $\mathbb{R}$  with $supp \  m_f \subseteq f(X)$ such that 
\[ m_f (W)  = \mu(f^{-1}(W) \setminus \{0\} ) \text{   for every open set   } W \in \mathbb{R} , \]
thus, 
\[ m_f (W)  = \mu(f^{-1}(W)) \text{   for every open set   } W \in \mathbb{R} \setminus \{0\} . \]
If $ \mu$ is finite then 
\begin{eqnarray} \label{mfOpen}
 m_f (W)  = \mu(f^{-1}(W )) \text{   for every open set   } W \in \mathbb{R} . 
\end{eqnarray}
The measure $m_f$ is the Stieltjes measure associated with the function $F(t) = \mu(f^{-1}((t, \infty) \setminus \{0\})$ 
(respectively,  $F(t) = \mu(f^{-1}((t, \infty))$ if $ \mu$ is finite.)
(See Lemma 25  in Section 3 of \cite{Butler:QLFLC}.) \\
Define a quasi-integlral $\rho= \rho_{\mu}$ on  $C_c(X)$  by:
\begin{align} \label{rmformula}
\rho_{\mu} (f) = \int_{\mathbb{R}} id \  \, dm_f. 
\end{align} 
We also write $\rho_{\mu} (f)  =  \int_X f \ \, d\mu $.
If $\mu$ is a measure then  $ \rho_\mu (f) = \int_X f \ \, d\mu $ in the usual sense.
On singly generated subagerbras $\rho_\mu$ acts as follows: for every  $\phi \in C([a,b])$ (with $\phi(0) = 0$
if $X$ is locally compact but not compact)
\begin{align} \label{flamf}
 \rho_{\mu} (\phi \circ f) = \int_\mathbb{R} \phi \ d m_f = \int_{[a,b]} \phi \ d m_f, 
\end{align}
where $[a,b]$ is any interval containing $f(X)$.
(See Theorem 30 and Proposition 29 in Section 3 of \cite{Butler:QLFLC}.)
\item \label{mrSets3} 
Let  $\rho$ be a  quasi-integral on $C_c(X)$. The corresponding
compact-finite topological measure $ \mu = \mu_{\rho}$ is given as follows: \\
If $U$ is open, $ \mu_{\rho}(U) = \sup\{ \rho(f): \  f \in C_c(X), 0\le f \le 1, supp \  f \subseteq U  \},$\\
if $F$ is closed, $ \mu_{\rho}(F) = \inf \{ \mu_{\rho}(U): \  F \subseteq U,  U \in \mathscr{O}(X) \}$. \\
If $K$ is compact, $ \mu_{\rho}(K) = \inf \{ \rho(g): \   g \in C_c(X), g \ge 1_K \}  
= \inf \{ \rho(g): \   g \in C_c(X), 0 \le g \le 1, g=1 \mbox{  on   } K \}. $
(See Definition 33 and Lemma 35 in Section 4 of  \cite{Butler:QLFLC}.)
\end{enumerate}
\end{remark}

For examples of topological measures and quasi-integrals on locally compact spaces
see \cite{Butler:TechniqLC} and the last sections of \cite{Butler:TMLCconstr} and \cite{Butler:QLFLC}.

\begin{remark} \label{rhoAd}
If $f \cdot g =0, f,g \ge 0$ then $f,g$ are in the same singly generated subalgebra. 
If $\rho$ is a quasi-integral on a locally compact space, then  $ \rho(f) = \rho(f^+) - \rho(f^-) $ for any $f \in C_c(X)$.
If $f \cdot g =0$ then $ \rho(f+g) = \rho(f) + \rho(g).$  Also, $ \rho(0) =0$.
(See Lemma 19 and Lemma 20 in Section 3 of \cite{Butler:QLFLC}.) 
\end{remark}

The following is a part of Theorem 47 in Section 5 of \cite{Butler:QLFLC}.
\begin{theorem} \label{contCc}
Suppose $X$ is locally compact and $\rho$ is af quasi-linear functional  on $C_c(X)$.
If $f, g \in C_c(X), \, f,g \ge 0, supp \  f, supp \  g \subseteq K$, where $K$ is compact, then 
\[ | \rho(f ) -\rho(g) | \le \parallel f-g \parallel \, \mu(K), \]
where $ \mu$ is the compact-finite topological measure corresponding to $ \rho$.
In particular,  for any $ f \in C_c(X)$
\[ | \rho(f) | \le \parallel f \parallel \, \mu(supp \  f).\]
If $f, g \in C_c(X), \, supp \  f, supp \  g \subseteq K$, where $K$ is compact, then 
\[ | \rho(f ) -\rho(g) | \le 2  \parallel f-g \parallel \,  \mu(K) .\]
\end{theorem}

\section{Almost simple quasi-integrals} \label{Almsim}

A nontrivial topological measure assumes at least two values. Topological measures that assume exactly two values are important 
for proving results about repeated quasi-integration.
  
\begin{definition} \label{mulsimqi}
Let $X$ be locally compact. A topological measure is called simple if it only assumes values $0$ and $1$. 
A topological measure is almost simple if it is a positive scalar multiple of a simple topological measure. 
A quasi-linear functional is simple (almost simple) if the corresponding topological measure is simple (almost simple).
\end{definition}

\begin{lemma} \label{rhof1f21}
Suppose $X$ is locally compact and $\mu$ is a compact-finite topological measure that assumes more than two values. 
Let $ \rho$ be  the corresponding quasi-integral.
Then there are functions
$f_1, f_2 \in C_c^+(X) $ such that $f_1 f_2 = 0, \  \rho(f_1) = \rho(f_2) = 1$. Functions $f_1, f_2$ belong to the same subalgebra.
\end{lemma}

\begin{proof}
Choose compact $C_1$ such that $0 < \mu(C_1) < \mu(X)$. 
If $ \mu(X) < \infty$ by inner regularity of $\mu$ on $X \setminus C_1$ choose compact $C_2 \subseteq X \setminus C_1$  
such that $0 < \mu(C_2) < \mu(X)$. If $ \mu(X) = \infty$, i.e. $\mu(X \setminus C_1) = \infty$, 
choose a compact $C_2 \subseteq X \setminus C_1$ with $ \mu(C_2) > n$.
Let $U_1, U_2$ be open disjoint sets containing $C_1, C_2$. For $ \epsilon >0$ let  
$f_i \in C_c(X)$ be such that $  f_i =1 $ on $C_i, \ supp \  f_i \subseteq U_i$ and $ 0 < \mu(C_i) \le \rho(f_i) \le \mu(C_i) + \epsilon$ for $ i=1,2$.
Then $ f_1 f_2 = 0$  and by calibrating $f_i$ we may assume that $\rho(f_i) =1$ 
for $i=1,2$. (This proof is adapted from part of an argument in Theorem 1 in \cite{Grubb:Products}.)
The last statement follows from Remark \ref{rhoAd}.
\end{proof} 

The next theorem extends results for a simple quasi-state on $C(X)$ where $X$ is compact, given in Section 2 of  \cite{Aarnes:Pure}. 

\begin{theorem}  \label{rhosimple}
Let $X$ be locally compact. 
The following are equivalent for a quasi-integral $\rho$ on $ C_c(X)$:
\begin{enumerate}[label=(\roman*),ref=(\roman*)]
\item \label{sim} 
$\rho$ is simple. 
\item \label{ptm}
$m_f$ is a point mass at $y = \rho(f) \in f(X)$.
\item \label{phiout}
$ \rho(\phi \circ f) = \phi(\rho(f)) $ for any $\phi \in C([a,b])$ (with $\phi(0) = 0$ if $X$ is locally compact but not compact), where $f(X) \subseteq [a,b]$.
\item \label{mult}
$\rho$ is multiplicative on each singly generated subalgebra, 
i.e. for each  $f \in C_c(X)  $ we have:
$\rho(gh) =  \rho (g) \, \rho (h)$ for $g,h$ in the singly generated subalgebra $B(f)$. 
\end{enumerate}
\end{theorem}
 
\begin{proof}
\ref{sim} $\Longrightarrow$ \ref{ptm}.  
If $\rho$ is simple, i.e. the corresponding topological measure $\mu$ is simple, then
the measure $m_f$ in part \ref{prt1} of Remark \ref{RemART} is a point mass.  
From formula (\ref{rmformula}) we see that  
$m_f$ is a point mass at $y = \rho(f)$. To show
that $y \in f(X)$, suppose the 
opposite, and choose an open set $W$ such that $ y \in W, \ W \cap f(X) = \emptyset$. 
Then $m_f(W) = 1$, while $\mu(f^{-1} (W)) = \mu(\emptyset) = 0$, which contradicts
formula (\ref{mfOpen}). \\
\ref{ptm} $\Longrightarrow$ \ref{phiout}.
Since $m_f$ is a point mass at $y = \rho(f)  \in f(X)$, by formula (\ref{flamf})
\[ \rho( \phi \circ f) = \int_{[a,b]} \phi \, dm_f = \phi(y) = \phi (\rho(f)). \]
\ref{phiout} $\Longrightarrow$   \ref{mult}.
Let $ \phi \circ f, \, \psi \circ f \in B(f)$. Then $ \rho(( \phi \circ f) \, (\psi \circ f)) = \rho( ( \phi \, \psi)\circ f)) =
(\phi \,  \psi)( \rho(f)) = \phi(\rho(f)) \,  \psi(\rho(f)) = \rho(\phi \circ f) \,  \rho( \psi \circ f)$. \\
\ref{mult} $\Longrightarrow$ \ref{sim} 
Suppose to the contrary that $\rho$ is not simple, i.e. the corresponding compact-finite
topological measure $\mu$ is not simple.  Then  there is a compact set $K \subseteq X$ 
with $a= \mu(K) \in \mathbb{R}, \, a \neq 0,1$. 
Suppose first $a <1$.  Let $\epsilon >0$ be such that $a > (a + \epsilon)^2 $.
By part \ref{mrSets3} of Remark \ref{RemART} choose $ g \in C_c(X), g \ge 1_K$ with $ \rho(g) < a+\epsilon$.
Note that $g^2 \in B(g)$ and $g^2 \ge 1_K$. Then $ \rho(g^2) \ge \mu(K)=a > (a+\epsilon)^2 > (\rho(g))^2$, and 
$\rho$ is not multiplicative on  $B(g)$. \\
Assume now that $ a>1$. There is an open set $U$ (containing $K$) such that $\mu(U) =b \in \mathbb{R}$ and $ b >1$.
Let $\epsilon >0$ be such that 
$b< (b - \epsilon)^2 $. By part \ref{mrSets3} of Remark \ref{RemART} 
let $f \in C_c(X)$ be such that  
$0 \le f \le 1, supp \  f \subseteq U, $ and $b-\epsilon < \rho(f)$. Since $f^2 \in B(f)$  and $ supp \  f^2 \subseteq U$ we have
$\rho(f^2) \le b < (b - \epsilon)^2< ( \rho(f)) ^2 $, and $\rho$ is not multiplicative on $B(f)$. 
\end{proof}

\begin{theorem} \label{rhoas}
Let $X$ be locally compact. 
The following are equivalent for a quasi-integral $\rho$ on $ C_c(X)$:
\begin{enumerate}[label=(\roman*),ref=(\roman*)]
\item \label{as} 
$\rho$ is almost simple.
\item \label{SvistCon}
If $ f g = 0$ then $\rho(f) \rho(g) = 0$.
\item \label{SvistCon+}
If $ f g = 0, \, f,g \ge 0 $ then $\rho(f) \rho(g) = 0$.
\item \label{Con-}
If $ f g = 0, \, f,g \le 0 $ then $\rho(f) \rho(g) = 0$.
\end{enumerate}
\end{theorem}

\begin{proof}
\ref{as}   $\Longrightarrow$ \ref {SvistCon}.
If $ \rho$ is almost simple, the topological measure $ \mu$ corresponding to $ \rho$ assumes two values, and the measure
$m_f$  supported on $f(X)$  as in Remark \ref{RemART} has the form $m_f = c \delta_y$, where $ c >0, y \in f(X)$.  
For  $g$ the measure $m_g = k \delta_z$, where $k >0, z \in g(X)$. Then  $ \rho(f) \rho(g) = ck f(y) g(z) = 0$.   
\ref{SvistCon}   $\Longrightarrow$ \ref{as}.
Lemma \ref{rhof1f21} shows that the topological measure corresponding to $ \rho$ can not assume more than 2 values. Thus, $ \rho$ is almost simple. 
The equivalence of \ref{SvistCon+} and \ref{SvistCon} follows from $\rho(f) = \rho(f^+) - \rho(f^-), \rho(g) = \rho(g^+) - \rho(g^-) $ 
(see Remark \ref{rhoAd})
and the fact that the product of any two of $f^+, f^-, g^+, g^-$ is $0$. Equivalence of \ref{SvistCon+} and \ref {Con-} is obvious.
\end{proof}

\begin{remark}
If a compact-finite topological measure $\mu$ assumes only two values, but not $0,1$, then 
$ \mu$ is almost simple. Write $ \mu = c \mu',$ where $ \mu'$ is simple and $ c>0$. Then $ \rho = c \rho'$, where quasi-integrals $\rho$ and $ \rho'$
correspond to $ \mu$ and $ \mu'$. Quasi-integral $ \rho$ is no longer multiplicative on singly generated subalgebras.
\end{remark}

\section{Repeated quasi-integrals} \label{Products}

Let $\mu$ be a compact-finite topological measure on $X$ with corresponding quasi-integral $\rho$, 
and $\nu$ be a  compact-finite topological measure on $Y$ with corresponding quasi-integral $\eta$. 

\begin{definition}  \label{basicProd}
For a set $A$ in $X \times Y$ let $A_y = \{ x: \ (x,y) \in A\}$, and let $A_x = \{ y: \ (x,y) \in A\}$. 
\end{definition}

\begin{remark} \label{fibers}
If the set $A$ is closed/compact/open then so is the set $A_y$. 
Note also that $(A \setminus B)_y = A_y \setminus B_y$.
\end{remark}

Let $ f \in C_c (X \times Y)$. Define continuous functions $f_y$ on $X$ and $f_x$ on $Y$ by $f_y(x) = f(x,y) = f_x(y)$. 
Compact $C = \pi_1(supp \  f)$, where 
$\pi_1: X \times Y \rightarrow X$ is the canonical projection, contains $supp \  f_y$ for any $y$. 
We have $f_y \in C_c(X)$, and $f_x \in C_c(Y)$.

\begin{definition} \label{TrhoSeta}
Define real-valued functions $T_\rho (f)$ on $Y$ and $S_\eta (f)$ on $X$
by:  
$$T_\rho (f) (y) = \rho(f_y), \ \ \ \ \ \ \ \ S_\eta (f) (x) = \eta(f_x).$$
\end{definition}

\begin{proposition} \label{TrhoCc}
Suppose $X \times Y$ is locally compact and $ f \in C_c(X \times Y)$. If $\rho$ is a quasi-integral on $C_c(X)$ 
then $T_\rho (f) \in C_c(Y) $ and $ \parallel T_\rho (f) \parallel  \le  \parallel f \parallel \mu(\pi_1(supp \  f))$.  Similarly, if 
$\eta$  is a quasi-integral on $C_c(Y)$ then  $S_\eta (f) \in C_c(X)$ and 
$ \parallel S_\eta (f) \parallel  \le  \parallel f \parallel \nu(\pi_2(supp \  f))$. Here $\pi_1: X \times Y \rightarrow X$ and $\pi_2: X \times Y \rightarrow Y$ 
are canonical projections.
\end{proposition}

\begin{proof}

Let $y \in Y$. We shall show that $T_\rho (f)$ is a continuous function at $y$. 
Compact $C = \pi_1(supp \  f)$ contains  $ supp \  f_y$.
Let $ \epsilon >0$. 
For each $x \in C$ let $U_x$ be a neighborhood of $x$ and $V_{x,y}$ be a neighborhood of $y$ such that 
$| f(x,y) - f(x', y') | < \epsilon $ whenever  $(x',y') \in U_x \times  V_{x,y}$. Open sets $U_x$ cover $C$, so let 
$U_{x_1}, \ldots, U_{x_n}$ be a finite subcover, and let $V_y = \bigcap_{i=1}^n V_{x_i, y} $. Take $y' \in V_y$. 
Then for each $x \in C$ there is $i$ such that $x \in U_{x_i} $, so $(x, y') \in  U_{x_i} \times V_y \subseteq U_{x_i} \times V_{x_i,y} $, 
and then $ | f_y(x) - f_{y'} (x) | =| f(x,y) - f(x, y') | < \epsilon $.  For any $x \notin C$ we have $ f_y(x) = f_{y'} (x)=0$.
Therefore, $ \parallel f_y - f_{y'} \parallel < \epsilon$  (for any $y' \in V_y$).
Since $supp \  f_y, supp \  f_{y'} \subseteq C$, by Theorem \ref{contCc} for any $y' \in V_y$ we have:
$$ | T_\rho (f)(y) - T_\rho (f)(y') | = | \rho(f_y) - \rho(f_{y'} |  \le 2 \parallel f_y - f_{y'} \parallel  \mu(C)  < 2 \epsilon \mu(C),$$
and the continuity of $T_\rho (f)$ at $y$ follows.

From Theorem \ref{contCc}  for any $y$ we have $ | T_\rho (f) (y) | = | \rho(f_y)| \le \parallel f_y \parallel \mu(C)  \le \parallel f \parallel \mu(C). $
Thus, 
$$ \parallel T_\rho (f) \parallel \le \parallel f \parallel \mu(C) =   \parallel f \parallel \mu(\pi_1(supp \  f)).$$

Since $X$ and $Y$ are locally compact, for $x \in X$ let $U(x)$ be a 
neighborhood of $x$ such that $\overline{ U(x)}$ is compact in $X$, 
and let $V(y)$ be a neighborhood of $y \in Y$ such that $\overline{ V(y)}$ is compact in $Y$. 
Open sets $U(x) \times V(y)$ cover $supp \  f$, so let $U_1\times V_1, \ldots, U_n\times V_n$ be
a finite subcover of $supp \  f$. Let $G = \overline V_1 \cup \ldots \cup \overline V_n$, a compact in $Y$. 
For each $x \in X$ and each $ y \notin G$ we have$(x,y) \notin supp \  f$, so $f(x,y) = 0$.
This means that $f_y = 0$ for each $y \notin G$. Then $\rho(f_y) =0$, i.e. 
$T_\rho(f) (y) = \rho(f_y) = 0$ for each $ y\notin G$. Hence, $T_\rho (f) \in C_c(Y)$.
\end{proof}

We are now ready to define repeated quasi-integrals.
\begin{definition} \label{erre}
Using Definition \ref{TrhoSeta} and Proposition \ref{TrhoCc} 
we define the following functionals on  $C_c(X \times Y)$:  
\[ (\eta \times \rho) (f) =\eta( T_\rho(f)) = \int_Y T_\rho(f) \, d \nu, \]
\[  (\rho \times \eta) (f) = \rho (S_\eta (f)) = \int_X S_\eta (f) \,  d \mu.\]
\end{definition}

\begin{remark}
Functionals $\eta \times \rho$ and $\rho \times \eta$ are real-valued since $ \eta$ and $ \rho$ are quasi-integrals.
\end{remark}

\begin{lemma} \label{mainAlmsim}
Suppose $ \eta$ and $ \rho$ are quasi-integrals. Let $ \rho = c \rho'$, where $ c >0$. Then 
$\eta \times \rho$ is a quasi-integral iff $ \eta \times \rho'$ is a quasi-integral, and $ \eta \times \rho = c (  \eta \times \rho')$. Similarly, 
if $ \eta = k \eta', \, k>0$ then $ \eta \times \rho$  is a quasi-integral iff $  \eta' \times \rho$ is a quasi-integral, and $ \eta \times \rho= k (\eta' \times \rho)$. 
\end{lemma}

\begin{proof}
We have $T_{\rho}(f) (y)  =cT_{\rho'}(f) (y)$ for every $y$, i.e. $T_{\rho}(f)  = cT_{\rho'}$. Since $ \eta $ is homogeneous, 
we see that $\eta \times \rho= c ( \eta \times \rho')$. 
Since $\eta \times \rho$ and $\eta \times \rho'$ satisfy conditions \ref{QIconsLC}, \ref{QIlinLC}, and \ref{QIpositLC} of Definition \ref{QI} simultaneously, 
$ \eta \times \rho$ is a quasi-integral iff $\eta \times \rho'$ is. 
The proof of the last part is similar but easier. 
\end{proof}

We would like to know whether $\eta \times \rho = \rho \times \eta$. We shall see later that, unlike the case of linear functionals, this is not usually the case.

For $g \in C_c(X)$ and $h \in C_c(Y)$ let $(g \otimes h)(x,y) = g(x)h(y)$. Then 
$(g \otimes h) \in C_c (X \times Y)$.  

\begin{proposition} \label{gtimh}
\begin{enumerate}[label=(\arabic*),ref=(\arabic*)]
\item
 $T_\rho (g \otimes h)  = \rho(g) \ h$.
\item \label{nrgh}
$ (\eta \times \rho) (g \otimes h) = (\rho \times \eta) (g \otimes h)  =  \rho(g) \eta(h). $
\end{enumerate}
\end{proposition}

\begin{proof}
The proof is as in Proposition 1 in \cite{Grubb:Products}.
\end{proof}

\begin{remark} \label{munuXY}
If $\mu$ is the compact-finite topological measure corresponding to quasi-linear functional $\rho$, 
$\nu$ is the compact-finite topological measure corresponding to $ \eta$, 
let $\nu \times \mu$ be the compact-finite topological measure corresponding to $ \eta \times \rho$, and 
$\mu \times \nu$ be the compact-finite topological measure corresponding to $ \rho \times \eta$.
Using  part \ref{nrgh} of Proposition \ref{gtimh} and part \ref{mrSets3} of Remark \ref{RemART} it is easy to see that 
$$ \mu(X)  \nu(Y) \le (\nu \times \mu) ( X \times Y), \ \ \   \mu(X)  \nu(Y) \le  (\mu \times \nu) (X \times Y). $$
By Definition \ref{erre}, Theorem \ref{contCc}, and Proposition \ref{TrhoCc}
the opposite inequalities also hold, so
$$ (\nu \times \mu) ( X \times Y) =  \mu(X)  \nu(Y), \ \ \   (\mu \times \nu) (X \times Y) = \mu(X)  \nu(Y).$$
\end{remark}

\begin{remark} \label{HoPo}
It is easy to see that  if $ \eta$ and $ \rho$ are homogeneous, then so are  $\eta \times \rho$ and $\rho \times \eta$;
and that  if $ \eta$ and $ \rho$ are positive, then so are $\eta \times \rho$ and $\rho \times \eta$.  
Yet, $\rho \times \eta$ and $\eta \times \rho$ are not always quasi-integrals. The criterion for $\eta \times \rho$ to be a quasi-integral is given in 
Theorem \ref{nurhoqiTa} below. 
\end{remark}

\begin{theorem} \label{nurhoqiT}
\begin{enumerate}[label=(\arabic*),ref=(\arabic*)]
\item
If the compact-finite topological measure corresponding to $ \rho$ assumes more than two values and $\eta \times \rho$ is a quasi-integral, then 
$\eta$ is linear.
\item \label{simpleLin}
If $\eta$ is linear or $\rho$ is simple then $\eta \times \rho$ is a quasi-linear functional. 
\end{enumerate}
\end{theorem}

\begin{proof}
\begin{enumerate}
\item
Our proof follows part of Theorem 1 in \cite{Grubb:Products}. It is given for completeness and 
because intermediate results from the proof are needed elsewhere in the paper. 
Suppose that $\eta \times \rho$ is a quasi-integral and the compact-finite topological measure $ \mu$ corresponding to $ \rho$
assumes more than two values.
By Lemma \ref{rhof1f21} choose 
$f_1, f_2 \in C_c(X)$ such that  $\rho(f_1) = \rho(f_2) =1$ and  $f_1 f_2 = 0$. 
Take any $g,h \in C_c(Y)$. 
For any $ y \in Y$,  $ \ (g(y) f_1)(h(y) f_2) = 0,$ so 
using Remark \ref{rhoAd} we have:
\begin{align*}
T_\rho (f_1 \otimes  g &+ f_2 \otimes h)(y) =\rho( (f_1 \otimes  g + f_2 \otimes h)_y) = \rho(g(y) f_1 + h(y) f_2) \\
&= \rho(g(y) f_1)  + \rho(h(y) f_2) = g(y) \rho(f_1) + h(y) \rho(f_2), 
\end{align*}
i.e. 
\begin{align} \label{Trho}
T_\rho(f_1 \otimes g + f_2 \otimes h) = \rho(f_1) g + \rho(f_2) h = g + h. 
\end{align}
Since $\eta \times \rho$ is a quasi-integral and $(f_1 \otimes g)(f_2 \otimes h) =0$,  
using Proposition \ref{gtimh} and Remark \ref{rhoAd} we see that 
\begin{align*}
\eta(g) &+ \eta(h) =\rho(f_1) \eta(g) + \rho(f_2) \eta (h) \\ 
&=(\eta \times \rho)(f_1 \otimes g) + (\eta \times \rho) (f_2 \otimes h) 
= (\eta \times \rho) (f_1 \otimes g + f_2 \otimes h) \\
&= \eta (T_\rho (f_1 \otimes  g + f_2 \otimes h)) = \eta (g+h).
\end{align*}
Thus, $\eta$ is linear on $C(Y)$.

\item
To show that $\eta \times \rho$ is a quasi-integral,  by Remark \ref{HoPo} we need to 
check condition \ref{QIlinLC} of Definition \ref{QI}. 
Let  $ \phi, \psi \in C([a,b])$ where $[a,b]$ contains 
the ranges of $f$ and $T_\rho(f)$.
Recall that $f_y \in C_c(X)$ and note that $(\phi \circ f)_y = \phi \circ f_y$. In particular,
$\phi \circ f_y$ and $\psi \circ f_y$ are in the same subalgebra generated by $f_y$.
One may apply the argument from Theorem 1 in \cite{Grubb:Products} to prove the statement.
\end{enumerate}
\end{proof} 

\begin{theorem} \label{nurhoqiTa}
Suppose $ \eta$ and $ \rho$ are quasi-integrals.
$\eta \times \rho$ is a quasi-linear functional iff $ \eta$ is linear or $ \rho$ is almost simple.
\end{theorem}

\begin{proof}
($  \Longrightarrow$). Suppose $ \eta \times \rho$ is a quasi-integral. If $ \rho$ is not almost simple then $ \rho$ 
assumes more then two values, and by Theorem \ref{nurhoqiT} $ \eta$ must be linear. \\
($ \Longleftarrow$).
If $ \eta$ is linear then $ \eta \times \rho$ is a quasi-linear functional by Theorem \ref{nurhoqiT}.
Suppose $\rho = c \rho'$ , where $ \rho'$ is simple, $c >0$. Then by Theorem \ref{nurhoqiT} $\eta \times \rho'$ is a quasi-integral, and then by 
Lemma \ref{mainAlmsim} so is $\eta \times \rho$. 
\end{proof}

\begin{corollary} \label{nurhoqi}
Suppose $ \mu(X) = 1$.
The functional $\eta \times \rho$ is a quasi-integral iff $\eta$ is linear or $\rho$ is simple. 
\end{corollary}

\begin{remark}
We may phrase the results of Theorem \ref{nurhoqiTa} it terms of topological measures.
If $\mu$ is a compact-finite topological measure on $X$ and 
$\nu$ is a compact-finite topological measure on $Y$ we 
may define a compact-finite  product topological measure $\nu \times \mu$ on $X \times Y$ if either 
$\nu$ is a measure or $\mu$ is an almost simple topological measure. Similar results hold for topological 
measure $\mu \times \nu$ corresponding to quasi-integral $\rho \times \eta$. 
\end{remark}

\begin{theorem} \label{2then3}
Suppose $\rho, \ \eta$, and $\eta \times \rho$ are quasi-integrals. 
If  any two of them are simple, then so is the third one. 
\end{theorem}

\begin{proof}
Assume that $\rho$ and $\eta$ are simple. 
Let $f \in C_c(X \times Y)$ and $\phi \in C([a,b])$  (with $\phi(0)=0$ if $X$ is locally compact but not compact), where $[a,b]$ contains 
the ranges of $f$ and $T_\rho(f)$.
Since $\rho $ is simple, by part \ref{phiout} of Theorem \ref{rhosimple} we have 
$T_\rho(\phi \circ f) (y)  = \phi \circ T_\rho(f) (y)$ for all $y$, so $T_\rho(\phi \circ f)  = \phi \circ T_\rho(f)$. 
Since $\eta$ is simple, we have
\[ (\eta \times \rho)(\phi \circ f) = \eta(T_\rho(\phi \circ f) = \eta( \phi \circ T_\rho(f)) = \phi (\eta  \circ T_\rho(f)) = 
\phi ( (\eta \times \rho ) f). \]
By part \ref{phiout} of Theorem \ref{rhosimple} $\eta \times \rho$ is simple. (Cf. Corollary 1 in  \cite{Grubb:Products}.) 

Now assume that $\rho$ and $\eta \times \rho$ are simple. 
Take $ g \in C_c^+ (X)$ such  that $\rho(g) =1$. Since $\rho$ is multiplicative  on singly generated sublagebras, $\rho(g^2) = 1$. 
Suppose that $\eta$ is not simple. If $ \eta$ assumes two values then $\eta$  is almost simple, but not simple;
by Lemma \ref{mainAlmsim} $ \eta \times \rho$ is almost simple, but not simple, which contradicts our assumption.
So $ \eta$ assumes more than 2 values.
By Lemma  \ref{rhof1f21} let $h_1, \ h_2 \in C_c^+(Y)$ be such that $ \eta(h_1) = \eta(h_2) =1$, $h_1 h_2=0$, so $ \eta(h_1h_2)=0$.
Since  $(g \otimes h_1)\cdot ( g \otimes h_2) =0$, by Remark \ref{rhoAd} functions $g \otimes h_1, \, g \otimes h_2$ 
belong to the same singly generated subalgebra.  Since $ \eta \times \rho $ is simple, it is multiplicative on this subalgebra and then
using Proposition \ref{gtimh} we have:
\begin{align*}
(\eta \times \rho) (g^2 \otimes h_1 h_2) &=(\eta \times \rho)((g \otimes h_1)\cdot ( g \otimes h_2)) \\ 
&=((\eta \times \rho)(g \otimes h_1)) \cdot ((\eta \times \rho) (g \otimes h_2)) \\ 
&= \rho(g) \eta(h_1) \rho(g) \eta(h_2) = 1.
\end{align*}
On the other hand,  by Proposition \ref{gtimh} 
\[ (\eta \times \rho) (g^2 \otimes h_1 h_2) =\rho(g^2) \eta(h_1 h_2) =\eta(h_1 h_2) =0, \]
and we arrive at contradiction. Thus,  $\eta$ must be simple.

The proof  that $\rho$ is simple when $\eta$ and $\eta \times \rho$ are is similar. 
\end{proof}

From Theorem \ref{2then3} and Lemma \ref{mainAlmsim} we obtain

\begin{corollary}
If $\rho$ and $ \eta$ are almost simple, then so is $ \eta \times \rho$.
\end{corollary}

\begin{theorem}  \label{nrsimiff} 
Suppose $\rho, \ \eta, \ \eta \times \rho$ are quasi-integrals. 
Then $\eta \times \rho$ is simple iff  $\eta \times \rho = \eta' \times \rho'$  for some simple quasi-integrals $\rho'$ and $\eta'$.
\end{theorem}

\begin{proof}
If $\rho$ and $\eta$ are simple, then $\eta \times \rho$ is simple by Theorem \ref{2then3}. 

Now assume that $ \eta \times \rho$ is  simple.
First, we shall show that $\eta$ assumes exactly two values.
If  $ \eta$ attains more than two values,
by Lemma  \ref{rhof1f21} let $h_1, \ h_2 \in C_c^+(Y)$ be such that $ \eta(h_1) = \eta(h_2) =1$, $h_1 h_2=0$,
so $\eta(h_1 h_2) =0$.
Choose $ g \in C_c(X)$ such that $ \rho(g) >0$. 
Since  $(g \otimes h_1)\cdot ( g \otimes h_2) =0$, by Remark \ref{rhoAd} functions $g \otimes h_1, \, g \otimes h_2$ 
belong to the same singly generated subalgebra.  Since $ \eta \times \rho $ is simple, as in the proof of Theorem \ref{2then3} we have:
$$ (\eta \times \rho) (g^2 \otimes h_1 h_2) =\rho(g) \eta(h_1) \rho(g) \eta(h_2) = (\rho(g))^2 > 0. $$
On the other hand, 
\begin{align*}
(\eta \times \rho) (g^2 \otimes h_1 h_2) =\rho(g^2) \eta(h_1 h_2)  =0.
\end{align*}
This contradiction shows that $\eta$  assumes only two values. 

Write $ \eta = k \eta '$ for some $k$ and some simple quasi-integral $ \eta'$.
Then $ \eta \times \rho = (k \eta' ) \times \rho = \eta' \times (k\rho)$. Since $ \eta \times \rho$ and $ \eta'$ are simple, 
by Theorem \ref{2then3} so is $ \rho' = k \rho$, and the assertion follows.
\end{proof}

\section{Products of topological measures and Fubini's theorem} \label{Fubini}

The next two theorems describe how $\nu \times \mu$ acts on sets. 

\begin{theorem} \label{nmmumeas}
Let $\mu$ be a compact-finite topological measure on $X$ and $\nu$ a finite measure on $Y$.
Then for $U$ open in $X \times Y$ we have:
\[ (\nu \times \mu) (U) = \int_Y \mu(U_y) d\nu (y).\]
If  $ \mu$ is finite, for a compact set $K$ in $X \times Y$ we also have 
\[ (\nu \times \mu) (K) = \int_Y \mu(K_y) d\nu (y).\]
\end{theorem}

\begin{proof}
Part \ref{mrSets3} of Remark \ref{RemART} allows us to use the same argument as in Claim 2 and Claim 3 of Theorem 2 in  \cite{Grubb:Products} to prove
that for any open $U \subseteq X \times Y$ we have  $ (\nu \times \mu) (U) =  \int_Y \mu(U_y) d\nu(y).$
 
If $\mu$ is finite, then by Remark \ref{munuXY} so is $\nu \times \mu$ and
$(\nu \times \mu) ( X \times Y) = \mu(X) \nu(Y)$. 
If $ K \subseteq X \times Y$ is compact, by Definition \ref{TMLC} we have:
\begin{align*}
(\nu \times \mu) (K) &= (\nu \times \mu) (X \times Y) - (\nu \times \mu) ( (X \times Y) \setminus K) =  \\
&= \nu(Y) \mu(X) - \int_Y \mu((X \times Y) \setminus K)_y \, d \nu(y) \\
&= \nu(Y) \mu(X) - \int_Y( \mu(X)- \mu(K_y)) \,  d \nu(y) =  \int_Y \mu(K_y)  \, d \nu(y).  
\end{align*}
\end{proof}

\begin{theorem} \label{mnmusimpl}
Let $\mu$ be a simple topological measure on $X$ and $\nu$ be a compact-finite topological measure on $Y$.
Then for $U$ open in $X \times Y$ we have:
\[ (\nu \times \mu) (U) =  \nu(\{y: \mu(U_y) = 1\}).\]
If $ \nu$ is finite, for $K$ compact in $X \times Y$ we also have:
\[ (\nu \times \mu) (K) =  \nu(\{y: \mu(K_y) = 1\}).\]
\end{theorem}

\begin{proof}
For $A$ in $X \times Y$ let $B(A) = \{ y : \mu(A_y) = 1\}$. 
Let $U$ be open in $X \times Y$.
It is not hard to show (for example, by applying the argument from Claim 1 of Theorem 3 in \cite{Grubb:Products})
that $B(U)$ is also open.
Since $\mu $ is simple, for any compact $D \subseteq X \times Y$ we have 
$ (B(D))^c = B(D^c)$,
and $B(D)$ is closed. 

Suppose that $ K \subseteq B(U)$ is compact. 
For each $y \in K , \ \mu(U_y) =1$, so there is a compact set $C(y) \subseteq U_y$ with
$\mu(C(y))=1$. Then $C(y) \times \{y\} \subseteq U$, so there are open sets 
$V(y), \, W(y)$ such that $ C(y) \subseteq V(y) \subseteq X, \ y \in W(y) \subseteq Y$, and $V(y) \times \overline{ W(y)}\subseteq U$.
Finitely many $W(y_1), \ldots, W(y_n)$ cover $K$. 
Set $D= \bigcup_{i=1}^n C(y_i)  \times \overline{ W(y_i)}$. Then $D$ is compact and $ D \subseteq U$. 
Choose $f \in C_c(X \times Y)$  such that  $1_D \le f \le 1, \ supp \  f \subseteq U$. 
Then for $y  \in K$, say, $ y \in W(y_i)$, we have 
$ f_y =1$ on $C(y_i)$, so $1=\mu(C(y_i)) \le \rho(f_y) = T_{\rho} (f) (y)$. 
Thus $T_{\rho} (f) =1$ on $K$. Then  by part \ref{mrSets3}  of Remark \ref{RemART}
\[ \nu(K) \le \eta(T_{\rho} (f)) = (\eta \times \rho)(f)  \le (\nu \times \mu) (U) .\]
Taking the supremum over $K \subseteq B(U)$  shows that  
$\nu(B(U))  \le  (\nu \times \mu) (U). $

Now we shall show that  $(\nu \times \mu) (U) \le \nu(B(U))$. 
Since $(\nu \times \mu) (U) = \sup \{ (\eta \times \rho)(f): \, f \in C_c(X \times Y), supp \  f \subseteq U, 0 \le f \le 1 \}$, it is enough to show that 
$ (\eta \times \rho)(f) \le  \nu(B(U))$ for $f \in C_c(X \times Y)$ such that $ supp \  f \subseteq U, 0 \le f \le 1$. 
Let  $C = supp \  f$. 
Pick an open set $V$ with compact closure such that
$C \subseteq V \subseteq \overline{V} \subseteq U$ (see, for example, \cite{Dugundji}, Chapter XI, 6.2).
Note that $ f_y \in C_c(X), supp \  f_y \subseteq V_y,  0 \le f_y \le 1$. Then 
\begin{align*}
 \{ y: T_{\rho} (f) (y) >0 \} &= \{ y:  \rho(f_y) > 0 \}  \subseteq \{ y: \mu(V_y) >0 \} \\
&= \{ y: \mu(V_y) =1\} = B(V) \subseteq B(\overline{V}) \subseteq B(U).
 \end{align*}
Since $B(\overline{V})$ is closed, we have $ supp \  T_{\rho} (f)  \subseteq  B(\overline{V}) \subseteq B(U)$. 
By Proposition \ref{TrhoCc} $ \parallel  T_{\rho} (f)  \parallel \le 1$. 
By part \ref{mrSets3} of Remark \ref{RemART} $(\eta \times \rho)(f) = \eta(  T_{\rho} (f) ) \le \nu(B(U))$.

Now we have: $ (\nu \times \mu) (U) =  \nu(\{y: \mu(U_y) = 1\})$ for any open set $U \subseteq X \times Y$.
When $ \nu$ is finite,  the formula for compact $K$ can be proved as in Theorem \ref{nmmumeas}. 
\end{proof}

\begin{corollary} \label{mnmusimpla}
Suppose $\mu = c \mu'$ is an almost simple topological measure on $X$, where $\mu'$ is simple and $ c>0$, and $\nu$ is a topological measure on $Y$.
Then for $A$  open  in $X \times Y$ we have:
\[ (\nu \times \mu) (A) =  c \, \nu(\{y: \mu(A_y) = c\}).\]
When $\nu$ is finite, the same also holds for compact sets in $X \times Y$.
\end{corollary}

\begin{proof} 
Using Lemma \ref{mainAlmsim} and Theorem \ref{mnmusimpl} we have:
$(\nu \times \mu) (A) =c (\nu \times \mu') (A) = c \nu(\{y: \mu'(A_y) = 1\}) =  c \nu(\{y: \mu(A_y) = c\})$.
\end{proof}

\begin{corollary}
Suppose $\mu$ is a compact-finite topological measure on $X$, $\nu$ is a finite topological  measure on $Y$, and
$\nu \times \mu$ is a compact-finite topological measure.
If $A \subseteq X$ and $B \subseteq Y$ are both
open then $(\nu \times \mu) ( A \times B) = \mu(A) \nu(B)$.
When $\mu$ is finite, the same also holds when $A \subseteq X$ and $B \subseteq Y$ are compact sets.
\end{corollary}

\begin{proof}
By Theorem  \ref{nurhoqiTa}  either $ \mu$ is  almost simple or $ \nu$ is a finite measure. 
The corollary now  follows  from Theorem \ref{nmmumeas} and Corollary \ref{mnmusimpla}, and 
generalizes Remark \ref{munuXY}.
\end{proof}

\begin{lemma} \label{simpNotM}
If $\mu$ and $\nu$ are almost simple but not measures then $\eta \times \rho \neq \rho \times \eta$. 
\end{lemma}

\begin{proof}
We first prove (as in Corollary 3 in \cite{Grubb:Products})  that if
 $\mu$ and $\nu$ are simple but not measures then $\eta \times \rho \neq \rho \times \eta$. 
If $\mu$ and $\nu$ are not measures, by Theorem \ref{subaddit}
they are not subadditive, and
we may find open sets $U, V \subseteq X$ with $\mu(U) = \mu(V) =0, \  \mu(U \cup V) =1$, and
compact sets $C, K \subseteq Y$ with $\nu(C) = \nu(K) =0, \ \nu(C \cup K) = 1$. Taking complements
of $C$ and $K$ we get open sets $W, E \subseteq Y$ such that $\nu(W) = \nu(E) = 1, \  
\nu(W \cap E) =0$.
Let $A = (U \times W) \cup (V \times E)$.  We shall show that  $(\nu \times \mu)(A) =0$ using 
Theorem \ref{mnmusimpl}.   If $ y \in E \cap W$ then $A_y = U \cup V$, so $\mu(A_y) =1$. 
For the cases $ y \in W \setminus E, \ y \in E \setminus W, \ y \notin E \cup W$ we have $\mu(A_y) = 0$.
Then $(\nu \times \mu) (A) = \nu(E \cap W) = 0$. A similar argument shows that 
$(\mu \times \nu)(A) =1$. Since $(\nu \times \mu) (A) \neq  (\mu \times \nu)(A)$, we have
$\eta \times \rho \neq \rho \times \eta$. 

If $\mu$ and $\nu$ are almost simple but not measures, write $\mu = c \mu', \nu = k \nu' $, 
where  $c,k>0$ and $\mu', \nu'$ are simple, but not measures. With simple quasi-integrals $ \rho', \eta'$ corresponding to
$\mu', \nu' $ we have $ \eta \times \rho = ck (\eta' \times  \rho') \ne ck (\rho' \times \eta') = \rho \times \eta$.
\end{proof}

\begin{lemma} \label{ptmCase}
Suppose one of compact-finite topological measures $ \mu,\nu$ is a positive scalar multiple of a point mass. 
Then  $(\nu \times \mu) =  (\mu \times \nu)$.
\end{lemma}

\begin{proof}
If we prove the statement in the case when one of  $ \mu,\nu$ is a point mass, the lemma easily follows.  
So suppose $ \mu = \delta_{x_0}$. 
It is enough to show that $(\nu \times \mu) =  (\mu \times \nu)$ for open sets.
Let $U \subseteq X \times Y$ be open. Then $\mu(U_y) = 1$ iff $ (x_0, y) \in U$, and by 
Theorem \ref{mnmusimpl} we have:
\begin{align*}
(\nu \times \mu) (U) = \nu(\{y: \mu(U_y) = 1 \})  = \nu(\{ y: (x_0, y) \in U \}) = \nu(U_{x_0}).
\end{align*}
By Theorem \ref{nmmumeas} we have:
$$ (\mu \times \nu) (U) =  \int_X \nu(U_x) d \mu(x) = \nu(U_{x_0}).$$
Thus, $(\nu \times \mu) =  (\mu \times \nu)$.
\end{proof}

Now we are ready to answer the question of when a version of Fubini's theorem holds for repeated quasi-integrals, 
in other words, when $\eta \times \rho = \rho \times \eta$. 

\begin{theorem} \label{qiFubini}
Let $\rho$ be a  quasi-integral with  corresponding compact-finite topological measure $\mu$ on $Y$ and 
let $\eta$ be a  quasi-integral with  corresponding compact-finite topological measure $\nu$ on $Y$. 
Then  $\eta \times \rho = \rho \times \eta$ if and only if $\mu, \ \nu$ are both measures or 
one of $\mu, \ \nu$ is a positive scalar multiple of a point mass.
\end{theorem}

\begin{proof}
(${\Longleftarrow}$) 
If both $\mu, \ \nu$ are measures, then  $\eta \times \rho = \rho \times \eta$ by Fubini's theorem. 
If one of $\mu, \ \nu$ is a  positive scalar multiple of a point mass, then $\eta \times \rho = \rho \times \eta$ by Lemma \ref{ptmCase}. \\
(${\Longrightarrow}$)
First we shall show that $\eta \times \rho$ and $\rho \times \eta$ must be quasi-integrals. 
Suppose $\eta \times \rho$ is not a quasi-integral, that is, $ \rho$ is not almost simple and $\eta$ is not linear by 
Theorem \ref{nurhoqiTa}. 
Let $f_1, f_2 \in C_c(X)$ be functions given by Lemma \ref{rhof1f21}, 
so $f_1 f_2 = 0$, 
$\rho(f_1) = \rho(f_2) =1$. Also choose $g,  h \in C_c(Y)$ such that 
$\eta(g+ h) \neq \eta(g) + \eta(h)$. Let $f = f_1 \otimes g + f_2 \otimes h$. 
As in formula (\ref{Trho}), we have $T_{\rho} (f) = g + h$. Then 
\[ (\eta \times \rho)(f) = \eta(T_{\rho} (f)) = \eta(g + h).\]
Now we shall look at $(\rho \times \eta)(f) = \rho(S_{\eta}(f)).$
For each $x$,  $S_{\eta}(f) (x) = \eta(f_x) = \eta( f_1(x)  g + f_2(x) h).$
Since  $ (f_1(x)  g)  ( f_2(x) h) = 0$, we have:
\[ \eta( f_1(x)  g + f_2(x) h) = \eta (f_1(x)  g) + \eta (f_2(x) h) = f_1(x) \eta(g) + f_2(x) \eta(h). \]
Thus, $S_{\eta}(f)  =  f_1 \eta(g) + f_2 \eta(h)$. Since $f_1 \eta (g) \cdot f_2 \eta(h) =0$, we have
\[ (\rho \times \eta)(f) =\rho(S_{\eta}(f)) = \eta(g) \rho(f_1) + \eta(h) \rho(f_2) = \eta(g) + \eta(h).\]
and we see that $\eta \times \rho \neq \rho \times \eta$. Thus, $\eta \times \rho$ must be a quasi-integral, and so must $\rho \times \eta$. 

Both $\eta \times \rho$ and $\rho \times \eta$ are quasi-integrals. 
From Theorem \ref{nurhoqiTa}
we see that this happens only when (a) both $\mu$ and $\nu$ are measures,  or (b) at least one of $\mu$ or $\nu$ 
is a positive scalar multiple of a point mass, or  (c) both $ \mu $ and $\nu$ are almost simple, but not measures.
The first two cases produce $\eta \times \rho = \rho \times \eta$, by Fubini's Theorem or Lemma \ref{ptmCase}.
In the last case (c),  by Lemma \ref{simpNotM} $\eta \times \rho \neq \rho \times \eta$. This finishes the proof.
\end{proof}

\begin{remark}
As in the compact case (see \cite{Grubb:Products}), we have the following interesting phenomenon: 
if  $ \mu$ and $\nu$ are almost simple, but not measures, then  $\nu \times \mu$ and $\mu \times \nu$ 
are different, even though they agree on rectangles. This holds even when $X =Y$ and $ \mu = \nu$.
Taking $a (\mu \times \nu) + (m-a) (\nu \times \mu)$, where $m = \mu(X)^2$, we obtain uncountably many 
topological measures that agree on rectangles, but are distinct. This is impossible for measures on product spaces, 
as they are determined by values on rectangles.
\end{remark}

{\bf{Acknowledgments}}:
The author would like to thank
the Department of Mathematics at the University of California Santa Barbara 
for its hospitality and supportive environment.



\end{document}